\documentclass[11pt]{amsart} %

\usepackage{float}

\usepackage{array,amsmath, enumerate,  url, psfrag}
\usepackage{amssymb, amsaddr, fullpage}
\usepackage{graphicx,subfigure}
\usepackage{color}

\restylefloat{figure}

\newtheorem{theorem}{Theorem}[section]
\newtheorem{lemma}[theorem]{Lemma}
\newtheorem{conj}[theorem]{Conjecture}

\newtheorem{observation}[theorem]{Observation}

\newenvironment{definition}[1][Definition]{\begin{trivlist}
\item[\hskip \labelsep {\bfseries #1}]}{\end{trivlist}}

\title{Planar graphs without $4$-cycles and close triangles are $(2,0,0)$-colorable}

\author{Heather Hoskins$^{1}$ \hskip 0.15in Runrun Liu$^2$ \hskip 0.15in Jennifer Vandenbussche$^{3}$ \hskip 0.15in Gexin Yu$^{1,2}$}

\address{
$^1$\small Department of Mathematics, The College of William and Mary, Williamsburg, VA, 23185, USA.\\
$^2$\small Department of Mathematics and Statistics, Central China Normal University, Wuhan, Hubei, China.\\
$^{3}$\small Department of Mathematics, Kennesaw State University, Marietta, GA, 30060, USA}

\thanks{The research of the last author was partially supported by NSFC (11728102) and the NSA grant  H98230-16-1-0316.}

\email{jvandenb@kennesaw.edu, gyu@wm.edu}

\begin{document}

\maketitle

\begin{abstract}
For a set of nonnegative integers $c_1, \ldots, c_k$, a $(c_1, c_2,\ldots, c_k)$-coloring of a graph $G$ is a partition of $V(G)$ into $V_1, \ldots, V_k$ such that for every $i$, $1\le i\le k, G[V_i]$ has maximum degree at most $c_i$. We prove that all planar graphs without 4-cycles and no less than two edges between triangles are $(2,0,0)$-colorable.
\end{abstract}

\section{Introduction}
The coloring of planar graphs has a long history.  The well-known Four Color Theorem, proved by Appel and Haken (see \cite{AppelHaken1}-\cite{AppelHaken2}) in the 1970s, states that all planar graphs are 4-colorable.  Determining whether an arbitrary planar graph is $3$-colorable is NP-complete; much attention has been given to proving sufficient conditions under which planar graphs are 3-colorable.  The classic example is the theorem by Gr\"{o}tzch \cite{Grozch} showing that planar graphs without 3-cycles are 3-colorable.

Recently, the study of the coloring of planar graphs with 3 colors has taken a very interesting turn. Steinberg~\cite{S76} in 1976 famously conjectured that planar graphs without 4-cycles and 5-cycles are 3-colorable. Erd\H{o}s asked for the constant $D$ such that planar graphs excluding cycles of lengths from $4$ to $D$ are $3$-colorable.  Borodin, Glebov, Raspaud, and Salavatipour~\cite{BGRS05} showed that $D\le 7$.    After being open for almost 40 years, in a very recent paper~\cite{Cohen-Addad}, the Steinberg Conjecture was disproved by a counterexample.  This surprising result suggests that the property of planar graphs being 3-colorable may be more rare than was previously thought, and spurs the search for more classes of planar graphs that are 3-colorable.

One interesting restriction that gives rise to classes of 3-colorable planar graphs involves forbidding triangles that are close together.  This idea is illustrated in the famous conjecture by Havel.

\begin{conj}[Havel, 1969]
There is a constant $C$ (perhaps as small as $4$) such that any planar graph whose triangles are at distance at least $C$ from each other is $3$-colorable.
\end{conj}

This conjecture was resolved by Dvo\v{r}\'{a}k, Kr\'{a}l' and Thomas~\cite{Dvorak} by showing the truth for any planar graph $G$ with $d_{\Delta}(G)>10^{100}$, where $d_{\Delta}(G)$ is the length of the shortest path between the vertices of any two 3-cycles.  Clearly more work is needed to understand the constant $C$, but in the meantime, there have been advances that combine the hypotheses of the Steinberg and the Havel conjectures.  For example, Borodin and Glebov \cite{BorodinGlebov} showed that any planar graph $G$ without $5$-cycles and satisfying $d_{\triangle}(G) \geq 2$ is $3$-colorable.

With the recent counterexample to Steinberg's conjecture showing that it may be more difficult to find 3-colorable planar graphs than originally thought, it becomes more interesting to investigate ``nearly'' 3-colorable planar graphs.   A graph $G$ is $(c_1, c_2, \ldots, c_k)$-colorable if $V(G)$ can be partitioned into $k$ nonempty subsets $V_1, V_2, \ldots, V_k$ such that the maximum degree of $G[V_i]$ is at most $c_i$.  In other words, there exists a $k$-coloring such that for each color $i$, each vertex colored with $i$ has at most $c_i$ neighbors of the same color.   Clearly, a graph is properly $3$-colorable if and only if it is $(0,0,0)$-colorable.  In \cite{Cowen},  it is shown that every planar graph is $(2,2,2)$-colorable.

There are many results in this area; we refer interested readers to~\cite{MO15}.  As an illustration, the following is a list of results known for 5-cycle-free planar graphs.

\begin{theorem}
Let $G$ be a planar graph without $5$-cycles.
\begin{itemize}
\item If $G$ also has no $4$-cycles, then it is $(2,0,0)$- and $(1, 1, 0)$-colorable (\cite{Chen, XuMiaoWang, HillYu}).
\item If $G$ has no intersecting triangles, then it is $(2,0,0)$- and $(1,1,0)$-colorable (\cite{LLY15a, LLY15b}).
\item If $G$ has no $K_4^-$, then it is  $(1,1,1)$- and $(1, 1, 0)$-colorable. (\cite{HLY17, Xu}).
\end{itemize}
\end{theorem}

In ~\cite{WX13}, Wang and Xu proved that planar graphs without $4$-cycles are $(1,1,1)$-colorable (in fact, $(1,1,1)$-choosable), and constructed a  non-$3$-colorable planar graph that has no $4$-cycles (and $d_{\Delta}=1$).  (See Figure~\ref{4cyclefree}.)  Furthermore, although the result by Borodin and Glebov \cite{BorodinGlebov}  (and other likewise results) forbids 5-cycles and not 4-cycles, their proof involves showing that there are no internal $4$-cycles in a minimal counterexample.

\begin{figure}[H]
\includegraphics[scale=0.18]{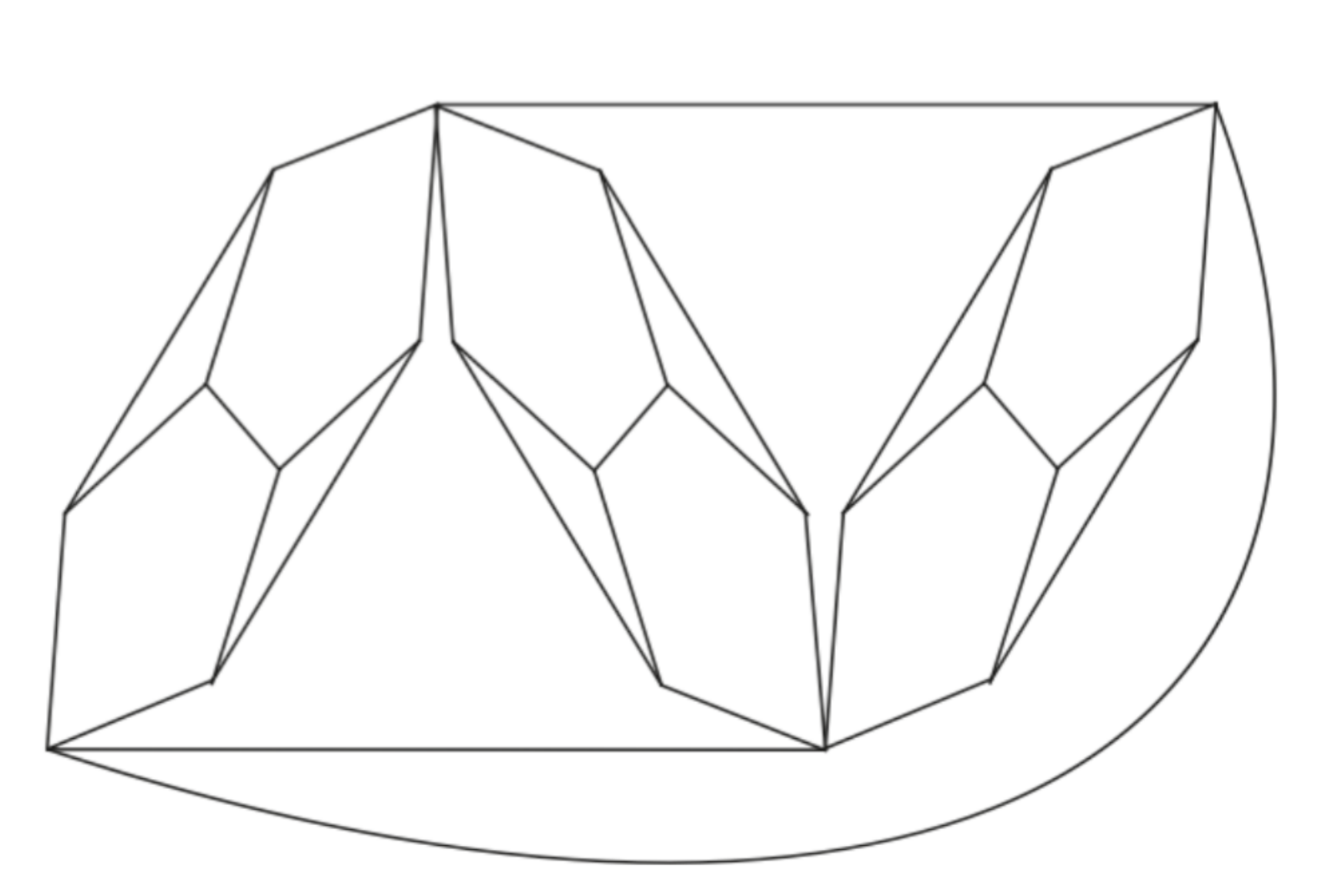}
\label{4cyclefree}
\caption{A non-$(0,0,0)$-colorable planar graph without $4$-cycles.}
\end{figure}

This motivates the study of the $3$-colorability of planar graphs without $4$-cycles (but perhaps with $5$-cycles) and satisfying $d_{\triangle}(G) \geq 2$. We conjecture that the following is true.

\begin{conj}\label{mystery_conjecture}
If $G$ is a planar graph without 4-cycles such that $d_{\triangle}(G) \geq 2$, then $G$ is 3-colorable.
\end{conj}

In this paper, we prove a relaxation of Conjecture~\ref{mystery_conjecture}.
Let $\mathcal{G}$ be the set of planar graphs with $d_{\triangle}(G)\ge 2$ and no 4-cycles.

\begin{theorem}\label{200colorable}
If $G \in \mathcal{G}$, then $G$ is $(2, 0, 0)$-colorable.
\end{theorem}

To prove Theorem~\ref{200colorable}, we use the idea of superextendable colorings introduced by Xu in \cite{Xu}.

\begin{definition}
A $(2,0,0)$-coloring $\phi$ of a subgraph $H$ of $G$ {\it superextends} to $G$ if there exists a $(2,0,0)$-coloring $\phi_G$ of $G$ that extends $\phi$ with the property that $\phi(v)\not=\phi(u)$ whenever $v\in H$ and $u\in G\cap N(v)-H$, where $N(v)$ is the set of neighbors of $v$.  We say that a subgraph $H\subseteq G$ is {\it superextendable} to $G$ if every $(2,0,0)$-coloring $\phi_H$ of $H$ superextends to $G$.   When we wish to specify $G$, we will say $(G,H)$ is superextendable.
\end{definition}

We need the following definition.

\begin{definition}
A 6-cycle is {\it bad} if alternating vertices along the 6-cycle are matched to the vertices of a triangle. The triangle is called an {\em interior triangle} of a bad $6$-cycle.   (See Figure~\ref{bad6cycle}.)  Otherwise, a 6-cycle is {\it good}.
\end{definition}

\begin{figure}[H]
\includegraphics[scale=0.08]{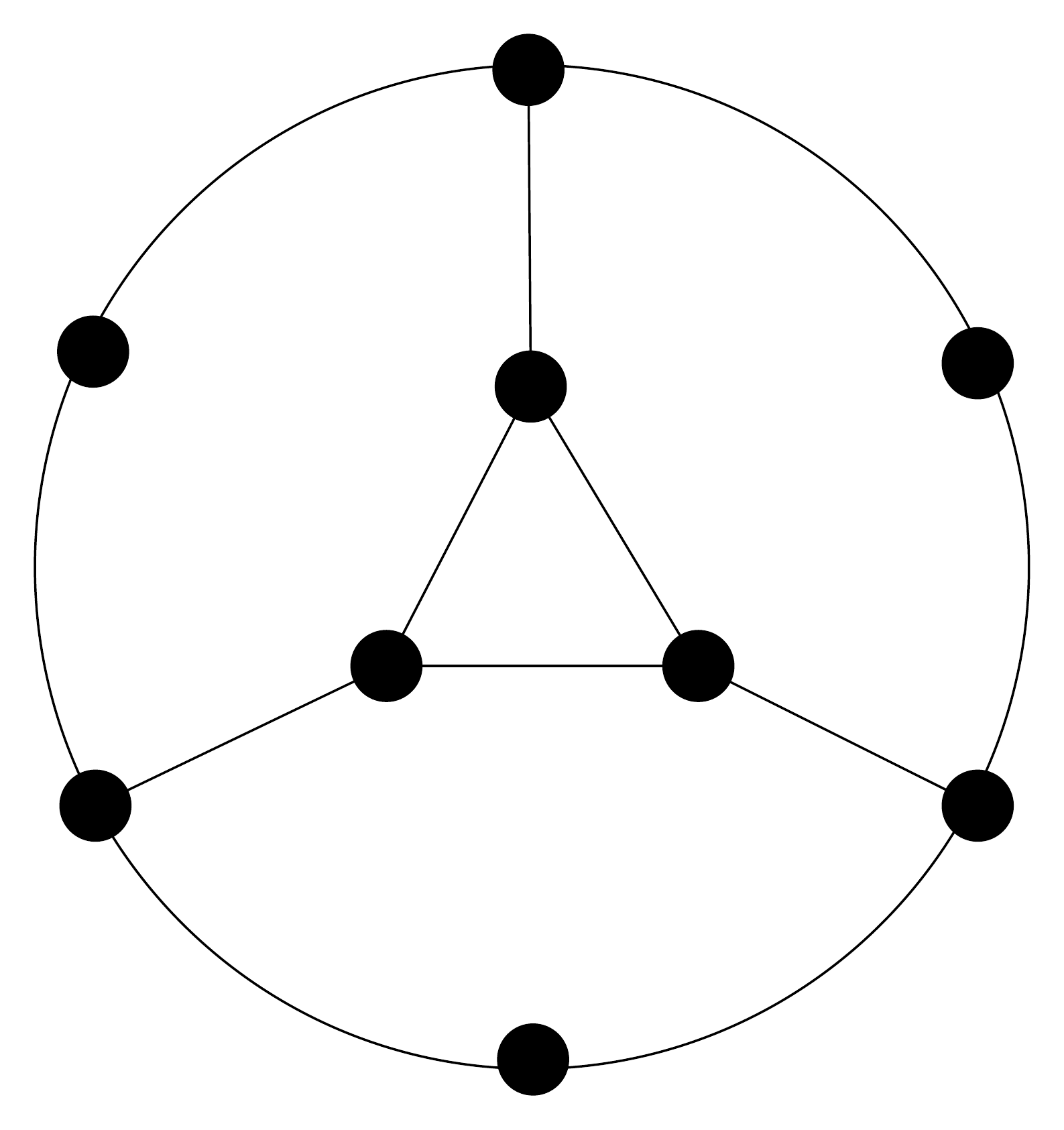}
\caption{A bad $6$-cycle.}
\label{bad6cycle}
\end{figure}

Our approach is to prove the following stengthening of Theorem~\ref{200colorable}.
\begin{theorem}
\label{Supext}
For each $G \in \mathcal{G}$, every triangle, 5-cycle and good 6-cycle in $G$ is superextendable.
\end{theorem}

Observe that the restriction to good 6-cycles is necessary. Otherwise, the graph in Figure~\ref{bad6cycle} is a counterexample: precolor the vertices of degree 3 on the 6-cycle with color 1.

Assuming Theorem~\ref{Supext} holds, it is easy to verify that Theorem~\ref{200colorable} also holds.  If $G$ has no triangles, then $G$ is $(2,0,0)$-colorable (in fact, $(0,0,0)$-colorable by Gr\"{o}tzch's theorem). Otherwise, fix a $(2,0,0)$-coloring $\phi$ of some triangle; by Theorem~\ref{Supext}, the coloring can be superextended to $G$, which is a $(2,0,0)$-coloring of $G$.

In Section~\ref{preliminaries}, we highlight the advantage of proving Theorem~\ref{Supext}, and we present some preliminary observations about a minimum counterexample to the theorem. The proof uses the discharging method, so Sections~3 and 4 contain, respectively, the reducible configurations and the discharging arguments.

\section{Preliminaries and Definitions}\label{preliminaries}

The advantage of proving the stronger theorem involving superextendable colorings was noted by Xu in \cite{Xu}.  Let a cycle $C$ in a plane graph $G$ be a {\it separating} cycle if the deletion of $C$ results in a disconnected graph. Let $int(C)$ denote the {\it interior} of $C$, and similarly $ext(C)$ the {\it exterior}, when the vertices of $C$ are deleted.  If a proper coloring of a separating cycle $C$ can be extended to $int(C)$ and $ext(C)$ individually, then the union of the two colorings is a proper coloring.  However, this property would not hold for $(2,0,0)$-colorings of the two subgraphs; a vertex of $C$ precolored with color 1 may have two neighbors of color 1 in both $int(C)$ and $ext(C)$, so the union of the two colorings would contain a vertex of color 1 with four neighbors of color 1.  The superextendable property allows us to combine colorings of $int(C)$ and $ext(C)$ into a $(2,0,0)$-coloring of the entire graph.

In order to illustrate this more clearly, we must introduce some notation that will be used for the remainder of the paper.  Our proof of Theorem~\ref{Supext} is by contradiction, so we will let $(G, C)$ for $G\in\mathcal{G}$ be a counterexample to the theorem of minimum order.  That is, some fixed precoloring $\phi$ of a cycle $C$ of length 3 or 5 or a good cycle of length 6 in $G$ cannot be superextended to $G$, and $G$ is the smallest graph with this property.  Let $V(C)$ denote the vertices of the cycle, and let $|V(C)|=r$.

We first observe that $C$ cannot be a separating cycle, analogous to Lemma 1 in \cite{Xu}.  Otherwise, $\phi$ can be superextended individually to $int(C)$ and $ext(C)$ by the minimality of $G$, and then the union of these two colorings would be a superextension of $\phi$ to $G$, a contradiction. Hence we may assume that $G$ is drawn with $C$ as the exterior face.

\begin{lemma}\label{NoSeparating}
$G$ does not contain separating triangles, 5-cycles or good 6-cycles.
\end{lemma}

\begin{proof}
Suppose otherwise that $G$ contains a separating cycle $C'$, where the length of $C'$ is 3 or 5, or $C'$ is a good 6-cycle.  Let $G_1$ be the subgraph of $G$ induced by $C'$ together with $ext(C')$, and $G_2$ the subgraph of $G$ induced by $C'$ together with $int(C')$.  Note that $C$ is contained in $G_1$.  By the minimality of $G$, $\phi$ superextends to a $(2,0,0)$-coloring $\phi_{G_1}$ of $G_1$.  Now $\phi_{G_1}$ restricted to $C'$ is a $(2,0,0)$ precoloring of $C'$, and again applying the minimality of $G$, it superextends to a $(2,0,0)$-coloring of $G_2$.  The union of these two colorings is a superextension of $\phi$ to $G$, a contradiction.
\end{proof}

The lack of separating short cycles provides additional information about the structure of $C$.  The next lemma follows~\cite{Xu}.

\begin{lemma}
\label{ChordsAndNeighbors}
The cycle $C$ is chordless, and for nonadjacent $x,y\in C$, $N(x)\cap N(y)\subseteq V(C)$.
\end{lemma}
\begin{proof}
The conclusion is trivial if $r=3$; suppose that $r=5$ or $r=6$.

If $C$ has a chord and $r=5$, then the chord separates $C$ into a 3-cycle and a 4-cycle, contradicting $G \in \mathcal{G}$. If $r=6$, then the chord would create either a 4-cycle and a 5-cycle, again contradicting $G \in \mathcal{G}$, or a 3-cycle and a 5-cycle. In the second case, since $V(C) \neq V(G)$, one of these cycles would have to be a separating cycle, contradicting Lemma~\ref{NoSeparating}.

Now consider nonadjacent $x,y$ in $V(C)$. Suppose that there exists $v\in N(x)\cap N(y)$ where $v\not\in V(C)$. If $r=5$, then $C$ together with $xvy$ forms a 4-cycle and a 5-cycle, which is impossible.  If $r=6$, then $C$ together with $xvy$ forms either a 6-cycle and a 4-cycle, which is impossible, or two 5-cycles.  Neither of the 5-cycles can be separating, but then $V(G)=C \cup \{v\}$, and it is easy to verify that such a graph is not a counterexample.  Therefore $x$ and $y$ can have no such neighbor.
\end{proof}

Now we introduce some definitions we use in the rest of the paper.  In a $(2, 0,0)$-coloring of $G$, a vertex $v$ is {\em $1$-saturated} if it is colored with $1$ and has two neighbors of color $1$; otherwise, it is called {\em nicely colored} (i.e.,  it is colored with $2$ or $3$, or it is colored with $1$ but not $1$-saturated).  If $v$ is nicely colored, then a neighbor of $v$ can be (re)colored with 1.  If $v$ has at most three colored neighbors, then {\it nicely recoloring $v$} means $v$ is either recolored with color 2 or 3 (if one of those colors is available), or $v$ has at most one neighbor with color 1, and $v$ remains color 1.

\begin{lemma}
Suppose $\phi_{G'}$ is a superextension of $\phi$ to $G' \subset G$, and $v \in V(G')$.  If $v$ is nicely recolored, then the extension remains a superextension.
\end{lemma}

\begin{proof}
Since the color of $v$ is not changed to 1, $v$ can be color 1 only if $\phi_{G'}(v)=1$, in which case $v$ must not have a neighbor of color 1 on $C$.
\end{proof}

Suppose that $x$ is a vertex of a face $f$.  A neighbor $v$ of $x$ is an {\it outer neighbor} (with respect to $f$) if $v$ is not on $f$.  A $k^+$-vertex (or $k^-$-vertex) in $G$ is a vertex of degree at least (or at most) $k$.  { A {\em $(k_1, k_2, \ldots, k_t)$-cycle} is a $t$-cycle whose vertices' degrees are $k_1, \ldots, k_t$, respectively. }
A vertex is {\em triangular} if it lies on a 3-face, otherwise it is called {\em nontriangular}.

Let $f$ be a $3$-face in $int(C)$.  If $v\not\in V(f)$ is adjacent to a $3$-vertex $x$ on $f$, then $f$ is called a {\em pendant face to $v$},  and $x$ and $v$ are pendant neighbors to each other. %

A $3$-vertex $v$ on a $5$-face in $int(C)$ is called {\em special} if its two neighbors on the face are $4^-$-vertices.   If $x$ is a $5^+$-vertex with a special neighbor $v$, we call the pendant 5-face containing $v$ a {\it pendant special 5-face} to $x$.   (See Figure~\ref{fig:special}.)   A $3$-vertex in $int(C)$ is {\it potentially special} if two of its neighbors are $4^-$-vertices in $int(C)$. Note that a special vertex is also potentially special.

\begin{figure}[H]
\includegraphics[scale=0.14]{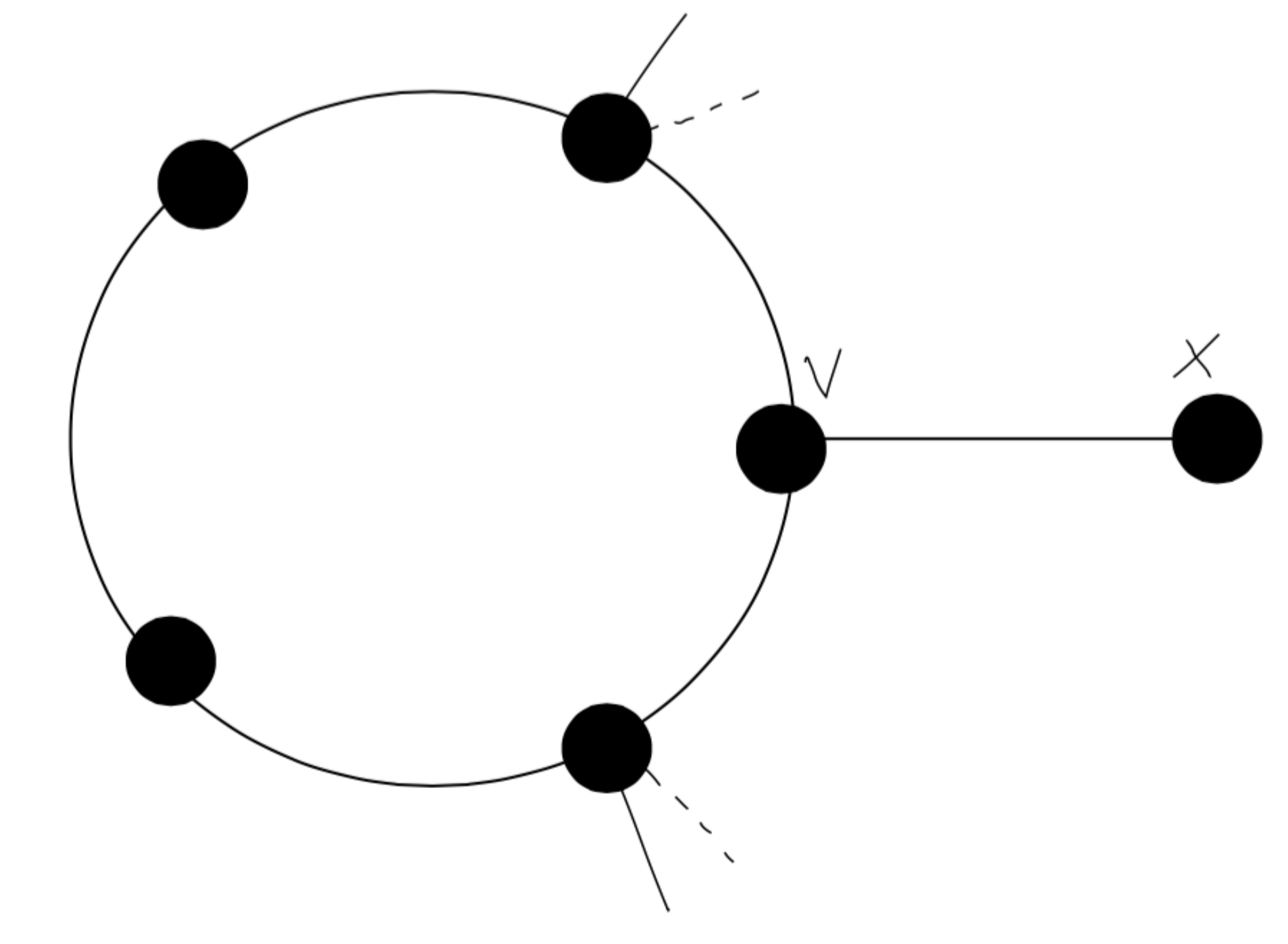}
\caption{Special vertex $v$ and a pendant special $5$-face to $x$.}
\label{fig:special}
\end{figure}

\section{Reducible Configurations}\label{ReducibleConfigurations}

\begin{lemma}
\label{No2V}
There are no $2^-$-vertices in $int(C)$.
\end{lemma}

\begin{proof}
Let $v$ be a $2^-$-vertex in $int(C)$.  Then $\phi$ superextends to $G-v$ by the minimality of $G$.   Now $v$ can be properly colored, a contradiction.
\end{proof}

The following lemma is a foundational one for our paper, and similar lemmas appear in other related results (see for example \cite{LLY15a}). We include the proof for completeness.

\begin{lemma}\label{3vert1}
Every 3-vertex in $int(C)$ is adjacent to at least one $5^+$-vertex or a vertex on $C$.
\end{lemma}

\begin{proof}
Let $v\in int(C)$ be a $3$-vertex not adjacent to a vertex of $C$. By the minimality of $G$, $\phi$ superextends to $G-v$.  Since this coloring cannot extend to $G$, the colors $1,2,3$ must appear in $N(v)$, and the vertex $u\in N(v)$ colored with $1$ is neither nicely colored, nor can it be recolored with $2$ or $3$.  Hence $u$ is adjacent to vertices colored with $2$ and $3$ and to two vertices of color $1$, thus $d(u)\ge 5$.
\end{proof}

\begin{lemma}
\label{2Colored3Face}
Suppose $\phi$ has been superextended to some subgraph of $G$, and let $v\in int(C)$ be a vertex that has exactly two colored neighbors, $v_1$ and $v_2$, both in $int(C)$.  Then $v$ can be recolored with color $1$, unless one of the following holds:
\begin{enumerate}
\item  $v_1$ and $v_2$ are adjacent and $d(v_1)+d(v_2)\ge 9$, or
\item  $v_1$ and $v_2$ are nonadjacent and one is a $5^+$-vertex.
\end{enumerate}
\end{lemma}

\begin{proof}
Suppose that $v$ cannot be recolored with $1$. Then if we recolor $v$ with $1$, some neighbor of $v$ (say $v_1)$ that is colored with $1$ will have three neighbors of color $1$.  Note that color $2$ and $3$ must appear in $N(v_1)$, for otherwise we recolor $v_1$ with the absent color, so the degree of $v_1$ is at least $5$.  Assume further that $v_1v_2\in E(G)$.  If $d(v_1)+d(v_2)\le 8$, then $d(v_1)=5$ and $d(v_2)=3$.  As there are at most two different colors in $N(v_2)$, we can properly recolor $v_2$ (if its color is $1$) or remove the color of $v_2$ and then properly recolor $v_1$ and $v_2$ in order(if its color is $2$ or $3$). Now in both cases, we can recolor $v$ with $1$.
\end{proof}

\begin{observation}
Let $v$ be a special or potentially special vertex with $N(v)$ in $int(C)$. By Lemma~\ref{3vert1}, $v$ must be adjacent to a $5^+$-vertex.  By Lemma~\ref{2Colored3Face}, in any precoloring of $G$ in which the $5^+$-neighbor $v$ has not been colored, $v$ can be recolored with 1.
\end{observation}

\begin{lemma}\label{Max1}
Let $v$ be a vertex in $int(C)$. If $v$ is adjacent to $m$ special vertices and $t$ pendant $(3,3,5^-)$- or $(3,4,4)$-faces, then $m+t\le d(v)-2$.
\end{lemma}

\begin{proof}
Suppose to the contrary, that there exists some vertex $v\in int(C)$ with $m+t> d(v)-2$. Consider $G'=G\backslash\{v\}$. We know that $G' \in \mathcal{G}$, so $\phi$ superextends to a coloring $\phi_{G'}$ of $G'$. By Lemma~\ref{2Colored3Face}, the $m+t$ pendant or special vertices in $N(v)$ can be recolored with color 1, leaving at most one vertex in $N(v)$ with a different color.   Now either color 2 or 3 is available to properly color $v$, superextending $\phi$ to $G$, a contradiction.
\end{proof}

\begin{lemma}\label{5Face}
Let $(v_1, \dots, v_6)$ be an internal $6$-cycle with chord $v_1 v_3 $.  Let $d(v_1)=d(v_3)=3$, and $v_4$ be adjacent to $k$ internal pendant $(3,3,5^-)$-, $(3,4,4)$-faces and special vertices.  If $d(v_2)\le5$ and $f=v_1v_2v_3$ is not the interior triangle of any bad $6$-cycle,  then $d(v_5) \geq 5$, and $k\le d(v_4)-3$.
\end{lemma}

\begin{proof}
Suppose the statement is not true.  Consider the graph $G'$ formed by deleting $v_1$ and $v_3$ and identifying $v_4$ and $v_6$ into a single vertex $X$.

We claim that $G' \in \mathcal{G}$. First of all, we do not create chords in $C$ of $G$ since both $v_4$ and $v_6$ are in $int(C)$. The only path of length 3 from $v_6$ to $v_4$ in $G$ goes through $v_1$ and $v_3$, else there would be a separating 5-cycle in $G$ or a triangle at vertex $v_4$ or $v_6$.  Hence no new triangles are created by identifying $v_4$ and $v_6$.  Further, since $v_1$ and $v_3$ are triangular in $G$, neither $v_4$ nor $v_6$ can be triangular. Thus $d_{\triangle}(G') \ge2$.   It remains to show that $G'$ has no 4-cycles; such a 4-cycle could be created by a path of length 4 in $G$ from $v_4$ to $v_6$.  For such a path to exist, there must be a 4-cycle containing $v_5$, which does not exist, or a separating 6-cycle in $G$. Since $G$ does not contain good separating 6-cycles and $f=v_1v_2v_3$ is not the interior triangle of any bad $6$-cycle, $G$ must contain a 6-cycle containing $v_4, v_5, v_6$ with a triangle inside, pendant to $v_5$. Let $x$ be the pendant neighbor to $v_5$.  But since there are no separating triangles or 5-cycles in $G$, $v_5$ and the two triangular neighbors of $x$ are all degree 3, contradicting Lemma~\ref{3vert1}.  Hence there is no such bad 6-cycle, and  $G' \in \mathcal{G}$.

By the minimality of $G$, we know that there exists a superextension $\phi_{G'}$ of $\phi$ to $G'$.  We claim we can extend $\phi_{G'}$ to a coloring $\phi_G$ that superextends $\phi$ to $G$, a contradiction. Let $\phi_G(x)=\phi_{G'}(x)$ for $x\in V(G')\backslash\{X\}$ and $\phi_G(v_4)=\phi_G(v_6) = \phi_{G'}(X)=\alpha$. Recolor $v_2$ so that it is nicely colored.  It remains to color $v_1$ and $v_3$, and {to verify that $v_5$ is adjacent to at most two neighbors of color $1$ when it is colored with $1$}.

If $\alpha=2$ (or symmetrically 3), then we can properly color $v_3$.  Since $v_1$ has three nicely colored neighbors, it can be colored, completing the superextension to $G$.   Hence we may assume $\alpha=1$.

Suppose that either $\phi_G(v_5) \neq 1$, or $\phi_G(v_5)=1$ and $v_4$ and $v_6$ are the only neighbors of $v_5$ that are colored with 1. At least one of $\{v_4, v_6\}$ is nicely colored with 1; assume by symmetry that $v_4$ is nicely colored. Properly color $v_1$, and now there is a color available for $v_3$, again a contradiction.  If $d(v_5)<5$, then $v_5$ can be recolored in this way, hence $d(v_5) \geq 5$.

Now suppose that $\phi_G(v_5)=1$ and $v_5$ has a third neighbor colored with 1.  (Since $X$ was adjacent to $v_5$ in $G'$, $v_5$ cannot have more than three neighbors with color 1.)  Observe that $v_3$ is a pendant neighbor of $v_4$.  Remove the color of $v_4$, and by Lemma~\ref{2Colored3Face}, recolor with color 1 the $k-1$ other neighbors of $v_4$ that are special vertices or pendant neighbors on $(3,3,5^-)$- or $(3,4,4)$-faces.  Since $d(v_4) \leq k+2$, $v_4$ has at most one neighbor colored from $\{2,3\}$.  Hence $v_4$ can be properly colored.  If $v_2$ is not colored with $1$, then  we can color $v_3$ with $1$ and properly color $v_1$.  If $v_2$ is colored with $1$, then $v_3$ and $v_1$ can be consecutively colored properly. Therefore $d(v_4) > k+2$.
\end{proof}

\begin{lemma}\label{two3face}
Let $v\in int(C)$ be a $5$-vertex with $N(v)=\{v_i:  1\le i\le5\}\subseteq int(C)$. Let $f_i$ be the face containing $v_ivv_{i+1}$ for $i=1,2$. If $v_1$ and $v_3$ are both $3$-vertices that are on internal $(3,3,5^-)$- or $(3,4,4)$-faces  and both $f_1$ and $f_2$ are $5$-faces in $int(C)$, then $d(v_2)\ge 4$.
\end{lemma}

\begin{proof}
Suppose otherwise, that $d(v_2)=3$. Then we discuss the following two cases.

{\bf Case 1:} One of $v_1$ and $v_3$ (say $v_1$, by symmetry) is not on the interior triangle of a bad $6$-cycle. Let $G'$ be the graph formed by identifying $v_2$ and $v_5$ in $G-v$ into vertex $X$. First of all, we do not create chords of $C$, for otherwise, the chord must be incident with $X$, conradicting $v_2$ and $v_5$ in $\int(C)$. Note that no new triangles can be created, else there would be a separating $5$-cycle in $G$, contradicting Lemma~\ref{NoSeparating}. Since $d_{\triangle}(G) \ge2$,  $v_2$ and its neighbors are all nontriangular, hence $d_{\triangle}(G') \ge2$. Also, $G'$ contains no $4$-cyces, else there would be a $4$-path between $v_2$ and $v_5$ which implies a separating good $6$-cycle in $G$, contradicting Lemma~\ref{NoSeparating}.  Therefore, $G' \in \mathcal{G}$. By the minimality of $G$, we know that there exists a superextension $\phi_{G'}$ of $\phi$ to $G'$. We show that $\phi_{G'}$ can be extended to a coloring $\phi_{G}$ of $G$.  Let $\phi_{G}(v_2)=\phi_{G}(v_5)=\phi_{G'}(X)$, and let $\phi_G(x)=\phi_{G'}(x)$ for $x \in V(G)-\{v, v_2, v_5\}$.  It remains to color $v$ to arrive at a contradiction. Recolor $v_1$ and $v_3$ with 1 by Lemma~\ref{2Colored3Face}. If $\phi_{G'}(X)=1$, then $v$ can be properly colored. If $\phi_{G'}(X)=2$ or $3$ (say 2), then $v_4$ must be colored $3$, else we can color $v$ properly. In this case, we can properly recolor $v_1$ and $v_3$, and then color $v$ with $1$.

{\bf Case 2:} Both $v_1$ and $v_3$ are on the interior $(3,3,3)$-faces of bad $6$-cycles. Let $N(v_i)=\{v,u_i,w_i\}$ for $i=1,2,3$ such that $w_1$ and $w_2$ are on $f_1$ and $u_2$ and $u_3$ are on $f_2$. Let $G'=G-\{v_3,u_3,w_3\}$. By the minimality of $G$, there exists a superextension $\phi_{G'}$ of $\phi$ to $G'$. We first claim that $v$ is $1$-saturated. For otherwise, color $v_3$ with $1$, and then either $\phi_{G'}(u_2)\not=1$ and we color $u_3$ with $1$ and then $w_3$ properly,
or $\phi_{G'}(u_2)=1$ and we color $w_3, u_3$ properly in order, a contradiction. Similarly, $u_2$ is also $1$-saturated. Furthermore, neither $v$ nor $u_2$ can be recolored, so their neighborhoods must have color set $\{1,1,2,3\}$.  Further, $v_1$ must be colored with $1$, or we could recolor it by Lemma~\ref{2Colored3Face}. We claim that $w_2$ must also be colored with $1$. For otherwise, we can recolor $w_1$ with 1 and then recolor $v_1$ properly, a contradiction. But since all neighbors of $v_2$ are colored with $1$,  $v_2$ can be recolored with a different color so that $v$ can be nicely recolored, a contradiction again.
\end{proof}

\begin{lemma}
\label{335Face}
Given a $(3,3,5^-)$-face $f$ in $int(C)$, the pendant neighbors of the 3-vertices on $f$ either are in $V(C)$ or have degree at least $5$.
\end{lemma}

\begin{proof}
Consider a 3-face $f=xyz$ in $int(C)$, where $d(z)\le5$ and $d(y)=d(x)=3$ (see Figure~\ref{cat}). Assume to the contrary that the outer neighbor $y'$ of $y$ has degree at most $4$, but $y' \notin V(C)$. Consider $G'=G\backslash\{x,y\}$. Because $G' \in \mathcal{G}$, we know that there exists a superextension $\phi_{G'}$ of $\phi$ to $G'$.  Since $z$ and $y'$ have degree at most 3 in $G'$, they can be nicely recolored.  But now we extend $\phi_{G}$ to $G$ by properly coloring $x$ and then coloring $y$, a contradiction.
\end{proof}


\begin{figure}[H]
\centering
\includegraphics[scale=0.18]{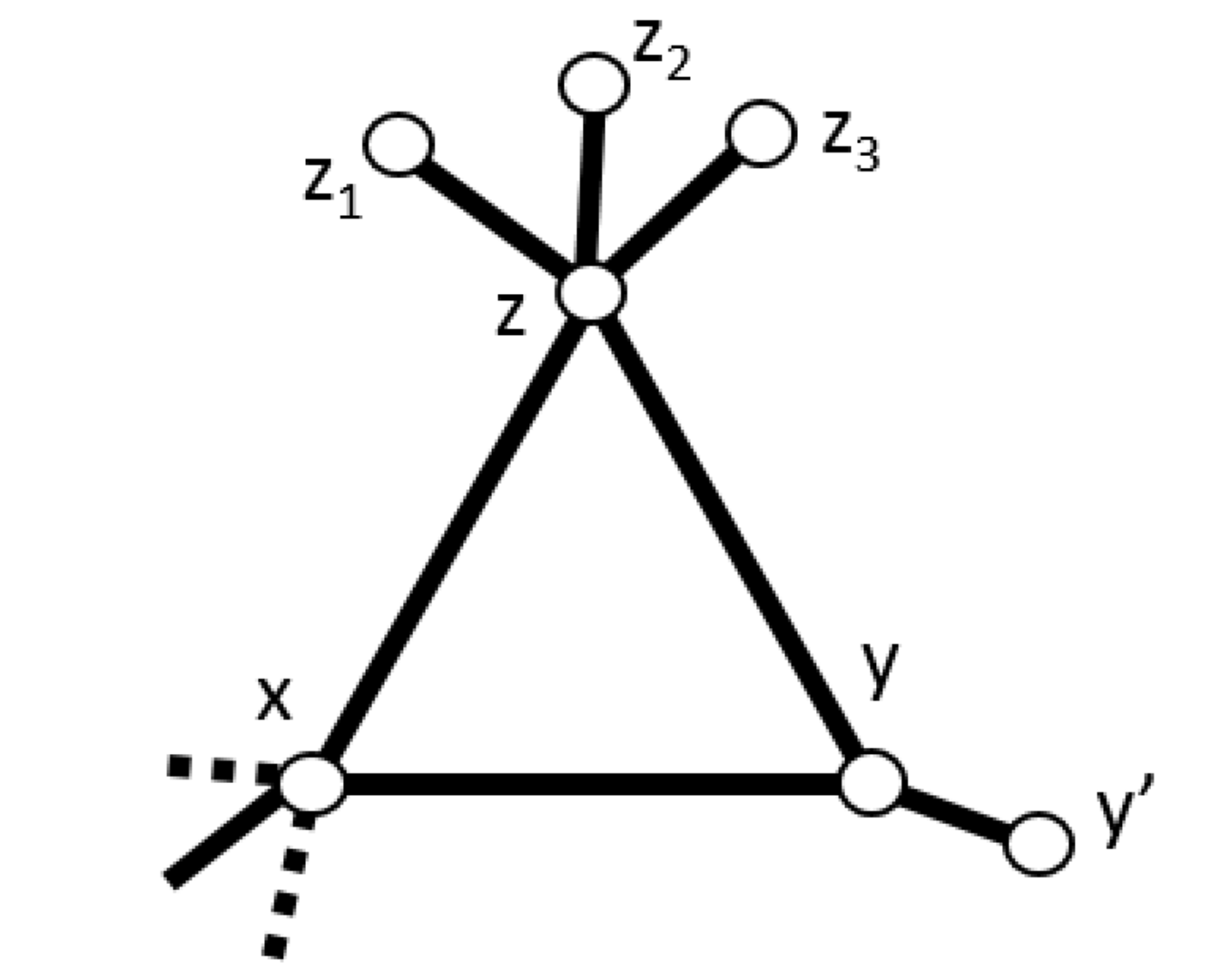} \hskip 1in
\includegraphics[scale=0.18]{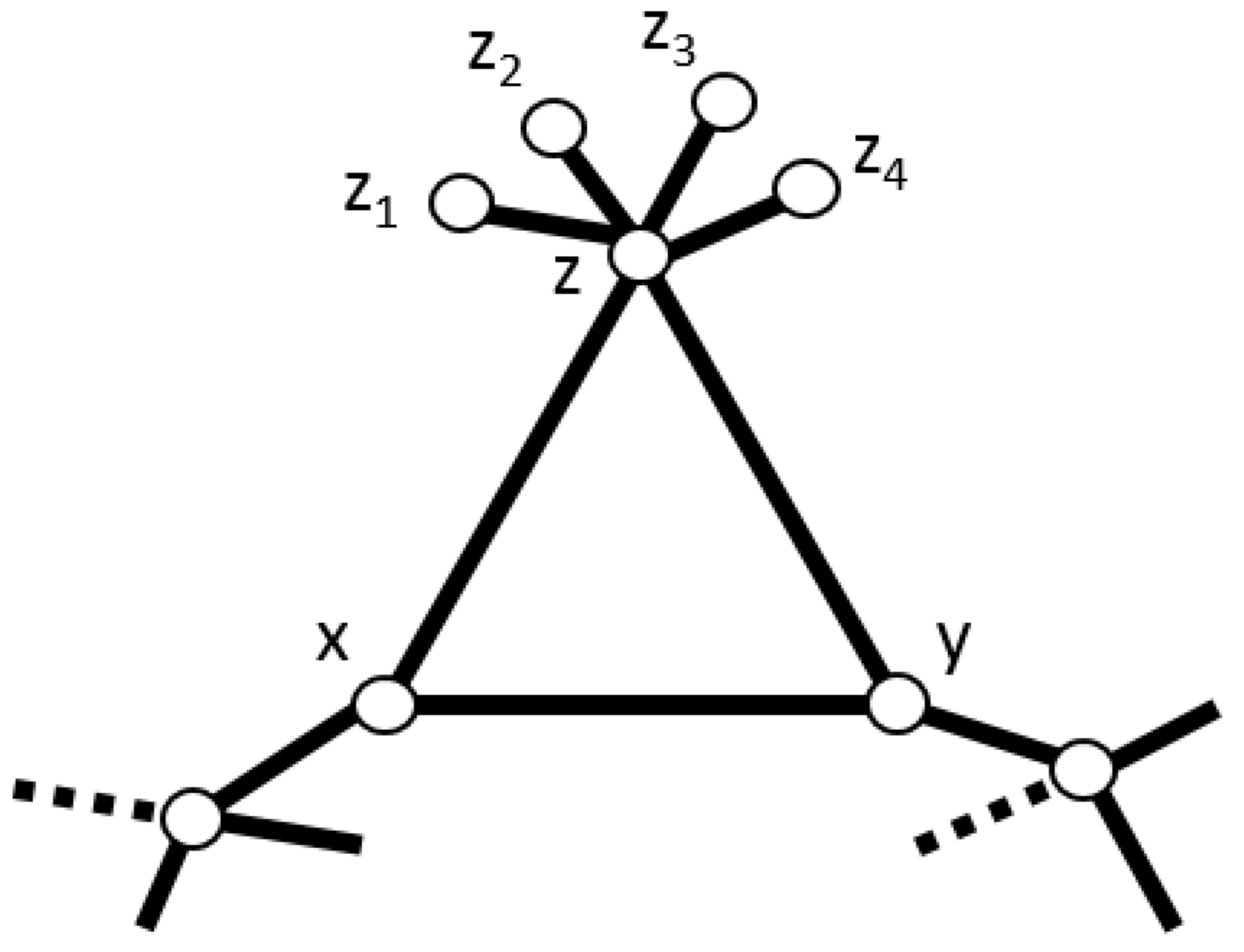}
\caption{$(3, 5^-, 5)$-face and $(3,3,6)$-face used in Lemmas~\ref{335Face}, ~\ref{335Face2} and ~\ref{336Face}}
\label{cat}
\end{figure}

\begin{lemma}\label{334Face}
Let $v\in int(C)$ be a $4$-vertex with $N(v)=\{v_i|1\le i\le4\}$ in the clockwise order. If $f=v_1vv_2$ is an internal $(3,3,4)$-face, then neither $v_3$ nor $v_4$ can be a $3$-vertex in $int(C)$.
\end{lemma}

\begin{proof}
Without loss of generality, let $v_3$ be a $3$-vertex in $int(C)$. Let $G'$ be the graph formed by identifying $v_2$ and $v_4$ into $X$ in $G-v$. First of all, we do not create chords in $C$ of $G$, for otherwise, the chord must be at $X$, thus there is a $3$-path connecting two vertices on $C$, which will create a separating $5$ or good $6$-cycle, a contradiction to Lemma~\ref{NoSeparating}. Note that no new triangles can be created, else there would be a separating $5$-cycle in $G$, contradicting Lemma~\ref{NoSeparating}. Since $d_{\triangle}(G) \ge2$, both $v_2$ and $v_4$ are nontriangular in $G-v$, hence $d_{\triangle}(G') \ge2$. None of $v_2,v,v_4$ is on a bad $6$-cycle, thus $G'$ contains no $4$-cycles, else there would be a separating good $6$-cycle in $G$, contradicting Lemma~\ref{NoSeparating}. Therefore, $G' \in \mathcal{G}$. By the minimality of $G$, we know that there exists a superextension $\phi_{G'}$ of $\phi$ to $G'$. We show that $\phi_{G'}$ can be extended to a coloring $\phi_{G}$ of $G$.  Let $\phi_{G}(v_2)=\phi_{G}(v_4)=\phi_{G'}(X)$, and let $\phi_G(x)=\phi_{G'}(x)$ for $x \in V(G)-\{v, v_2, v_4\}$. We claim this coloring extends to $v$, a contradiction.

If $\phi_{G'}(X)=2$ (or $3$), then we properly recolor $v_1$ and $v_3$ and color $v$ with 1. If $\phi_{G'}(X)=1$, then $N(v)$ has color set $\{1,1,2,3\}$, else we can color $v$ with the missing color. By symmetry, let $v_1$ be colored with $2$ and $v_3$ be colored with $3$.  We can recolor $v_1$ with either $3$ if the outer neighbor of $v_1$ is colored with $1$, or $1$ otherwise. In either case, $v$ can be colored with 2, a contradiction.
\end{proof}

\begin{lemma}
\label{335Face2}
Suppose that $f=xyz$ is a $(3,5^-,5)$-face in $int(C)$, with $d(x)\le5$, $d(y)=3$, and $d(z)=5$.  Let the outer neighbors of $z$ be $z_1, z_2, z_3$ in clockwise order so that $x$ and $z_1$ are on the same face.  Let $y'$ be the outer neighbor of $y$(See Figure~\ref{cat}).
\begin{enumerate}
\item At most one of $\{z_1,z_2,y\}$ (and symmetrically $\{z_1, z_3,y\}$) is potentially special (and hence at most one is special).
\item If  $z_2$ and $z_3$ are potentially special, then either $y'\in V(C)$, or  $d(z_1)\ge 5$.
\end{enumerate}
\end{lemma}

\begin{proof}
Consider the graph $G'$ formed by identifying vertices $x$ and $z_3$ into vertex $X$, and deleting the vertex $z$.  Note all 3-cycles in $G'$ were 3-cycles in $G$, else there would be a separating 5-cycle in $G$, contradicting Lemma~\ref{NoSeparating}. Also, since $z$ was incident to a 3-face, $z_3$ cannot be triangular, and hence $d_{\triangle}\ge2$ is maintained in $G'$. We also claim that $G'$ does not contain any 4-cycles.  Any such 4-cycle would correspond to a path of length 4 in $G$ between $x$ and $z_3$, and such a path would imply a separating 6-cycle in $G$; such a 6-cycle must be bad.  But since $x$ is triangular, the 6-cycle could not have another interior triangle, a contradiction. Hence $G'\in\mathcal{G}$, and by the minimality of $G$, $\phi$ superextends to a $(2,0,0)$ coloring $\phi_{G'}$ of $G'$. We show that $\phi_{G'}$ can be extended to a coloring $\phi_{G}$ of $G$ when the hypotheses fail.  Let $\phi_{G}(x)=\phi_{G}(z_3)=\phi_{G'}(X)=\alpha$, and let $\phi_G(v)=\phi_{G'}(v)$ for all other $v \in V(G)-z$.  It remains to color $z$ to arrive at a contradiction.

(1)  Assume first that $z_1$ and $z_2$ are both potentially special. Properly recolor $y$, and properly recolor $z_1$ and $z_2$.  If $\alpha=2$ (or symmetrically 3), then $z$ can be colored with 1, unless $z_1$, $z_2$, and $y$ are all colored with 1, in which case $z$ can be colored with $3$; even if $z_3 \in V(C)$, this would be a superextension of $\phi$ to $G$, a contradiction.  If $\alpha=1$, then $z_1$ and $z_2$ can be recolored with 1 by Lemma~\ref{2Colored3Face}, and $z$ can be properly colored.  Hence at most one of $\{z_1,z_2\}$ is potentially special. Repeating the proof with $z_3$ in place of $z_2$ shows that at most one of $\{z_1, z_3\}$ is potentially special.

Now suppose one of $\{z_1, z_2\}$ is potentially special (assume $z_1$ by symmetry), and $y$ is, as well.  Properly recolor $y$ and $z_1$.  If $\alpha=2$ and $z_2$ is colored with $2$ or $3$, then $z$ can be colored with 1.  Otherwise, $y$ and $z_1$ can be recolored with 1, and either 2 or 3 is available for $z$.

(2) Now assume that $z_2$ and $z_3$ are potentially special, $y' \notin V(C)$, and $d(z_1) \leq 4$.

If $\alpha=1$, then recolor $z_2$ with 1.  Now $z$ can be colored with $2$ or $3$ unless  $\phi_{G'}(z_1)\neq \phi_{G'}(y)$ and neither is color 1.  If $z_2$, $z_3$, or $x$ are not nicely colored with 1, then properly recolor them, and color 1 is now available for $z$.  If they are all nicely colored with 1, then $y$ can be recolored 1, unless $\phi_{G'}(y')=1$, in which case $y$ can be recolored with $\phi_{G'}(z_1)$.  In either case, $\phi_{G'}(y)$ becomes available for $z$, a contradiction.

If $\alpha=2$ (or symmetrically 3), consider the color on $z_1$.  If $\phi_{G'}(z_1)\neq 1$, then properly recolor $z_2$ and $y$, and now color 1 is available for $z$.  If $\phi_{G'}(z_1)=1$, then recolor $z_2$ with 1 by Lemma~\ref{2Colored3Face}. So $z$ can be colored with 3, unless $y$ is given color 3. In this case, consider $z_1$.  Since $d(z_1) \leq 4$, the vertex $z_1$ can be nicely recolored, and we can color $z$ with 1, a contradiction.
\end{proof}

\begin{lemma}\label{336Face}
Let $f$ be a $(3,3,6)$-face in $int(C)$ with vertices $x,y,z$ such that $d(z)=6$.  Then either a neighbor of $z$ is in $V(C)$, or $z$ has at most two potentially special neighbors.
\end{lemma}

\begin{proof}
Suppose that no neighbors of $z$ are in $V(C)$.  Let $z_1, z_2, z_3$, and $z_4$ be the outer neighbors of $z$, labeled as in Figure~\ref{cat}. Let $H_1$ be the graph formed by identifying $x,z_2,$ and $z_4$ in $G-\{z,y\}$ into a single vertex $X_1$, and  $H_2$ be graph formed by identifying $y,z_1,$ and $z_3$ in $G-\{z,x\}$ into a single vertex $X_2$.  Let $S_1=\{x, z_2, z_4\}$ and let $S_2=\{y, z_1, z_3\}$.  Assume by symmetry that the number of potentially special vertices in $S_1$ is at most the number of potentially special vertices in $S_2$.  This implies that we will consider $H_1$ for this proof, but a similar argument would hold for $H_2$ if $S_2$ had more potentially special vertices.

Note that all 3-cycles in $H_1$ were 3-cycles in $G$, else there would be a separating 5-cycle in $G$, contradicting Lemma~\ref{NoSeparating}. Also, since $z$ was incident to a 3-face, $d_{\triangle}\ge2$ is maintained in $H_1$. We also claim that $H_1$ does not contain any 4-cycles.   Any 4-cycle in $H_1$ would correspond to the contraction of the edges between two vertices in $\{x,z_2,z_4\}$, and that would imply a separating 6-cycle in $G$; such a 6-cycle must be good, since the outer neighbors of $z$ cannot be triangular, but no such separating cycle exists.  Thus $H_1\in\mathcal{G}$, and by the minimality of $G$, we know that $\phi$ superextends to a $(2,0,0)$-coloring $\phi_{H_1}$ of $H_1$.

We claim that $\phi_{H_1}$ extends to a $(2,0,0)$-coloring $\phi_G$ of $G$ that superextends $\phi$, a contradiction.  Let $\phi_G(v)=\phi_{H_1}(v)$ for $v\in H_1\backslash\{X_1\}$, and $\phi_G(x)=\phi_G(z_2)=\phi_G(z_4)=\phi_{H_1}(X_1)=\alpha$.  It remains to assign colors to $y$ and $z$.  Let $y'$ be the outer neighbor of $y$.  If $S_2$ contains at most one potentially special vertex, then by the minimality of $S_1$, the result holds.  Hence we may assume $S_2$ contains at least two potentially special vertices, and by symmetry, we may assume $z_1$ is potentially special.

Suppose first that $\alpha=1$.  Recolor $z_1$ with 1 by Lemma~\ref{2Colored3Face}.  If $\phi_{H_1}(z_3) \neq \phi_{H_1}(y')$, then $y$ can be colored with $\phi_{H_1}(z_3)$, leaving a color available for $z$.  If $\phi_{H_1}(z_3) = \phi_{H_1}(y')$, then $y$ can be colored with 1 and a color is left for $z$, unless $\phi_{H_1}(z_3) = \phi_{H_1}(y') = 1$ and $y'$ is improperly colored (and cannot be nicely recolored).  This implies that $y$ is not potentially special, and hence $z_3$ is potentially special.  Thus $z_3$ can be recolored with 1, and $y$ and $z$ can be colored with 2 and 3.

Otherwise, by symmetry we may assume that $\alpha=2$.  Properly color $y$, and then color 1 is available for $z$ unless either all of $S_2$ receives color 1, or some vertex in $S_2$ is not nicely colored with 1 and cannot be nicely recolored.  In the former case, color 3 is available for $z$.  In the latter case, some neighbor of $z$ in $S_2$ is not potentially special.  But then the other two vertices in $S_2$ must be potentially special, and they can be recolored with 1.  This leaves color 3 available for $z$.
\end{proof}

\section{Discharging Procedure}
\label{Discharging}

We are now ready to present a discharging procedure that will complete the proof of the theorem.  Let each vertex $v\in V(G)$ have an initial charge of $\mu(v)=2d(v)-6$, and each face $f\not=C$ in our fixed plane drawing of $G$ have an initial charge of $\mu(f)=d(f)-6$.  Recall that the length of $C$ is $r$; let $\mu(C)=r+6$. By Euler's Formula,   $\sum_{x\in V\cup F}\mu(x)=0$.

Let $\mu^*(x)$ be the charge of $x\in V\cup F$ after the discharge procedure. To lead to a contradiction, we shall prove that $\mu^*(x)\ge 0$ for all $x\in V\cup F$ and $\mu^*(C)>0$.

Let a $t$-face with exactly one vertex in $C$ be an $F_t'$-face, and a $t$-face with two or more vertices in $C$ be an $F_t''$-face for $t\in\{3,5\}$. Note that by Lemma~\ref{ChordsAndNeighbors}, no 3-face contains three vertices of $C$ and no 5-face contains four consecutive vertices of $C$.   Observe also that since $d_{\Delta} \geq 2$, a vertex can be incident to at most one 3-face.

We call a $5$-vertex $v$ {\em good} if it contains three consecutive neighbors that are neither special vertices on $5$-faces nor on internal pendant $(3,3,5^-)$- or $(3,4,4)$-faces of $v$, furthermore, they are the nontriangular neighbors when $v$ is on a $3$-face. Otherwise, it is {\em bad}. Extending this, a $4^+$-vertex in $int(C)$ is good if it is a nontriangular $4$-vertex, a good $5$-vertex, or a $6^+$-vertex. We call a $5$-face in $int(C)$ {\em rich} if it has one good $4^+$-vertex and two or more other $5^+$-vertices.

Below are the discharging rules:
\begin{enumerate}[(R1)]
\item \label{itm:4v} If $v$ is a 4-vertex and $f$ is an incident face in $int(C)$, then $v$:
\begin{enumerate}
\item \label{itm:4v_3f} gives $2$ to $f$ when $f$ is a $(3,3,4)$-face, and $\frac{5}{4}$ to $f$ when $f$ is any other triangular face.
\item \label{itm:4v_5f} gives $\frac{1}{2}$ to $f$ when $f$ is a 5-face and $v$ is nontriangular, and gives $\frac{1}{4}$ to $f$ when $f$ is a 5-face and $v$ is a triangular vertex with no incident $(3,3,4)$-face.
\end{enumerate}

\item \label{itm:5+v} If $v\in int(C)$ is a $d$-vertex with $d\ge 5$, then $v$:
\begin{enumerate}
\item \label{itm:5+v_inci5} gives $\frac{3}{8}$ to each incident $5$-face in $int(C)$ with exactly two $5^+$-vertices that are consecutive, gives $\frac{1}{3}$ to each incident $5$-face in $int(C)$ that is not rich and has at least three $5$-vertices, and gives $\frac{1}{2}$ to each other incident $5$-face, unless $v$ is a bad $5$-vertex and the $5$-face is rich, in which case $v$ gives $\frac{1}{4}$.  In addition, $v$ gives $\frac{1}{4}$ to each of its pendant special 5-faces in $int(C)$.

\item \label{itm:5+v_pend3f} gives $1, \frac{5}{8}, \frac{1}{2}$ to pendant $(3,3,3)$-, $(3,3,5)$-faces, and $(3,4^-,4)$-faces in $int(C)$, respectively.

\item \label{itm:5v} gives $\frac{7}{4}, \frac{3}{2}, 1$ to incident $(3,4^-,5)$-, $(3,5,5)$- and {other incident $3$-faces} in $int(C)$, respectively (when $d=5$).

\item \label{itm:6+v} gives $3, 2, 1$ to incident $(3,3,d)$-, $(3,4^+,d)$-, and $(4^+,4^+,d)$-faces in $int(C)$, respectively (when $d>5$).
\end{enumerate}

\item \label{itm:Cv} The initial charge of $r+6$ on $C$ is distributed as follows:
\begin{enumerate}
\item $C$ gets $2d(v)-6$ from each vertex $v\in C$, $1$ from each $7^+$-face.
\item $C$ gives $3$ to each $3$-face in $F_3'\cup F_3''$, $1$ to each $5$-face in $F_5'\cup F_5''$,  $1$ to pendant $(3,3,5^-)$- and $(3,4,4)$-faces in $int(C)$, and $\frac{1}{4}$ to its pendant special 5-faces.
\end{enumerate}
\end{enumerate}

\begin{lemma}
The face $C$ has a positive final charge.
\end{lemma}

\begin{proof}
Let $t_3, t_5$ be the number of pendant $3$-faces and pendant special $5$-faces at $C$, respectively.  Assume that $C$ gets $a$ from $7^+$-faces.  Let $E(C, V(G)-C)$ be the set of edges between $C$ and $V(G)-C$ and let $e(C, V(G)-C)$ be its size.  Then by (R3),
\begin{align}
\mu^*(C)&=r+6+\sum_{v\in C} (2d(v)-6)-3(|F_3'\cup F_3''|)-|F_5'\cup F_5''|-t_3-\frac{t_5}{4}+a \notag\\
&=r+6+2\sum_{v\in C} (d(v)-2)-2r-3(|F_3'\cup F_3''|)-|F_5'\cup F_5''|-t_3-\frac{t_5}{4}+a \notag\\
\label{calc}&=6-r+2e(C,V(G)-C)-3(|F_3'\cup F_3''|)-|F_5'\cup F_5''|-t_3-\frac{t_5}{4}+a.
\end{align}

We aim to balance the charge of 2 on each $e\in E(C,V(G)-C)$ with the charge distributed to the incident and pendant faces; we can view this as sharing a charge of 2 for each $e\in E(C,V(G)-C)$ with the faces.
\begin{enumerate}[(a)]
\item If $e$ is on a $3$-face $f\in F_3'\cup F_3''$, then $e$ can give $\frac{3}{2}$ to $f$, $\frac{1}{4}$ to a potential pendant $5$-face, and $\frac{1}{4}$ to a potential incident $5$-face.

\item If $e$ is adjacent to a pendant $3$-face, then it can give $1$ to the $3$-face and $\frac{1}{2}$ to each potential incident $5$-face.

\item If $e$ is neither on a $3$-face $f\in F_3'\cup F_3''$ nor adjacent to any $3$-face, then it can give $\frac{3}{4}$ to each potential incident $5$-face and $\frac{1}{4}$ to a potential pendant $5$-face.  In this case, $e$ would have a surplus of at least $\frac{1}{4}$.
\end{enumerate}

Observe first that pendant faces are collectively allocated $t_3 + \frac{t_5}{4}$ from $E(C, V(G)-C)$.  Since each face in $F_3'\cup F_3''$ contains two edges in $E(C,V(G)-C)$, it is allocated a charge of $3$.  Each face in $F_5'\cup F_5''$ contains two edges in $E(C,V(G)-C)$, and it is allocated at least $\frac{1}{2}\cdot 2 = 1$, unless it shares an edge with a $3$-face in $F_3'\cup F_3''$. In that case, it is not adjacent to pendant $3$-face, so it gains $\frac{1}{4}+\frac{3}{4}=1$. This implies that $$2e(C,V(G)-C)-3(|F_3'\cup F_3''|)-|F_5'\cup F_5''|-t_3-\frac{t_5}{4}\ge 0.$$

Hence from \eqref{calc}, $\mu^*(C)>0$ if $C$ is a $3$- or $5$-cycle.  When $r=6$, $\mu^*(C)\ge 0$, with equality only if $a=0$ and $$2e(C,V(G)-C)-3(|F_3'\cup F_3''|)-|F_5'\cup F_5''|-t_3-\frac{t_5}{4}=0.$$

This implies that each edge must be as in (a) or (b), and it is either the common edge of two $5$-faces and adjacent to a pendant $3$-face or the common edge of a $3$-face and a $5$-face and adjacent to a pendant $5$-face.  Note that edges on $3$-faces in $F_3''$ cannot be adjacent to pendant $5$-faces, so  $F_3''=\emptyset$.

Let $C=u_1u_2u_3u_4u_5u_6$. Suppose that $u_1$ is on a $3$-face in $F_3'$. Then $u_i$ must be a $2$-vertex for $2\le i\le 6$; otherwise $u_1$ and $u_i$ must be on the same $5$-face and thus $d_{\Delta}\ge 2$ implies $u_i$ cannot be on a 3-face or adjacent to a pendant $3$-face in $int(C)$. But in this case there is a $7^+$-face, contradicting $a=0$. Therefore $F_3'\cup F_3''=\emptyset$.

Now if $u_1, u_2$ are $3^+$-vertices and in the same $5$-face $u_1u_2v_2vv_1$, then both $v_1, v_2$ are in pendant triangles, and $v$ must be in both, contradicting $d_{\Delta} \geq 2$. Hence $C$ contains no consecutive $3^+$-vertices.  Let $d(u_1)\ge 3$ and $d(u_i)=2$ for $2\le i\le s-1$ and $d(u_s)\ge 3$. Note that $u_1,u_s$ are in the same $5$-face. Then $u_1, u_s$ have the same pendant $3$-face.  This implies that $u_s=u_3$. Again, there is no $7^+$-faces, so one of $u_4, u_5, u_6$ must be $3^+$-vertex. Therefore, by the above argument, it must be $d(u_5)\ge 3$ and $d(u_4)=d(u_6)=2$. But then $C$ is a bad $6$-cycle, contrary to our assumption that $C$ is good. 
\end{proof}

\begin{lemma}
Each face other than $C$ has nonnegative final charge.
\end{lemma}

\begin{proof}
Observe first that (R3a) is the only rule applied to $7^+$-faces; therefore all such faces have a nonnegative final charge.

Suppose next that $f$ is a face with $d(f)=3$; the initial charge on $f$ is $-3$.

If $f\in F_3'$, then $f$ gets 3 from the vertex in $V(C)$ incident to $f$ by (R3).  If $f\in F_3''$, then $f$ gets $\frac{3}{2}$ from the two $3^+$-vertices in $V(C)$ incident to $f$, again by (R3). In either case, $\mu^*(f) \geq -3+3=0$.

Next, suppose $f\in int(C)$.

\begin{itemize}
\item If $f$ is a $(3,3,3)$-face, then by Lemma~\ref{335Face}, its outer neighbors either have degree at least $5$ or lie on $C$.  Hence by (R\ref{itm:5+v_pend3f}) and (R3), $f$ gets 1 from each outer neighbor, and $\mu^*(f) \geq 0$.

\item If $f$ is a $(3,3,4)$-face, then $f$ gets $2$ from the incident 4-vertex by (R\ref{itm:4v_3f}).  Lemma~\ref{335Face} again guarantees that the outer neighbors of the 3-vertices on $f$ either have degree at least $5$ or lie on $C$, and hence $f$ gets at least $\frac{1}{2}$ from the outer neighbors of its 3-vertices by either (R\ref{itm:5+v_pend3f}) or (R3b). Hence $\mu^*(f) \geq -3+\left(2+\frac{1}{2} \cdot 2\right)=0$.

\item If $f$ is a $(3,3,5)$-face, then by (R\ref{itm:5v}) and (R\ref{itm:5+v_pend3f}), it gets $\frac{7}{4}$ from the 5-vertex and at least $\frac{5}{8}\cdot 2$ from the two outer neighbors, so $\mu^*(f) \geq -3+\left(\frac{7}{4}+\frac{5}{8} \cdot 2\right)=0$.

\item If $f$ is a $(3,4,5)$-face, then $f$ gets $\frac{5}{4}$ from its 4-vertex by (R\ref{itm:4v_3f}) and $\frac{7}{4}$ from its 5-vertex by (R\ref{itm:5v}), hence $\mu^*(f)\ge 0.$

\item If $f$ is a $(3,4,4)$-face, then $f$ gets $\frac{5}{4}$ from each incident 4-vertex by (R\ref{itm:4v_3f}) and at least $\frac{1}{2}$ from the pendant vertex , and $\mu^*(f)=0$.

\item If $f$ is a $(3,5,5)$-face or a $(3,3,6^+)$-face, then (R\ref{itm:5v}) or (R\ref{itm:6+v}), respectively, imply that $f$ receives a charge of 3 from its incident vertices.

\item If $f$ is a $(3,4^+,6^+)$-face, then $f$ receives 2 from the $6^+$-vertex by (R\ref{itm:6+v}) and at least $\frac{5}{4}$ from the $4^+$-vertex by (R\ref{itm:4v_3f}), (R\ref{itm:5v}) or (R\ref{itm:6+v}), and again, the final charge on $f$ is nonnegative.

\item If $f$ is a $(4^+,4^+,4^+)$-face, then by (R1) and (R2), $f$ gets at least 1 from each incident vertex.
\end{itemize}
Therefore the final charge on all 3-faces is nonnegative.

Assume now that $d(f)=5$, so the initial charge on $f$ is $-1$.  If $f$ is an $F_5'$- or $F_5''$-face, then by (R3b), $f$ gets $1$ from the incident vertices on $C$. Hence we let $f$ be a 5-face in $int(C)$.

Suppose $f$ contains at least three $5^+$-vertices. If $f$ is rich, then by (R2a), $f$ receives $\frac{1}{2}$ from the good $5^+$-vertex and at least $\frac{1}{4}$ from each of the other two (or more) $5^+$-vertices, and  $\mu^*(f)\ge -1+\frac{1}{2}+\frac{1}{4} \cdot 2=0$.  If $f$ is not rich, then $f$ receives $\frac{1}{3}$ from each and $\mu^*(f)\ge -1+\frac{1}{3} \cdot 3 = 0$.    Suppose $f$ contains exactly two non-consecutive $5^+$-vertices; then by (R2), $f$ gets $\frac{1}{2}$ from each, and $\mu^*(f)\ge 0$.  Similarly, since $f$ receives $\frac{1}{2}$ from each nontriangular $4$-vertex,  $\mu^*(f)\ge 0$ when $f$ has at least such $4$-vertices, or one such 4-vertex and one $5^+$-vertex.  Hence we may assume that $f$ contains at most two $5^+$- and nontriangular $4$-vertices, and when it has exactly two, they are consecutive $5^+$-vertices on $f$.

Let $f=v_1v_2v_3v_4v_5$.  By $d_{\Delta}(G) \geq 2$ and Lemma~\ref{334Face}, if $v_i$ is $4$-vertex on $f$ on an internal $(3,3,4)$-face in $int(C)$ (that is, $v_i$ is a 4-vertex that does not give $\frac{1}{4}$ to $f$), then either $v_{i-1}$ or $v_{i+1}$ is a nontriangular $4^+$-vertex.  Hence when $f$ has no $5^+$- or nontriangular $4$-vertices, $f$ also contains no 4-vertex on an internal $(3,3,4)$-face. Since any $3$-vertex on $f$ must be a special vertex, $\mu^*(f)\ge -1+\frac{1}{4}\cdot 5=\frac{5}{4}>0$ by (R\ref{itm:4v_5f}) and (R\ref{itm:5+v_inci5}). When $v_1$ is the only $5^+$- or nontriangular $4$-vertex on $f$, neither $v_3$ nor $v_4$ is a $4$-vertex on an internal $(3,3,4)$-face, and $f$ gets $\frac{1}{4}$ through each of $v_3,v_4$ and $\frac{1}{2}$ from $v_1$. Thus $\mu^*(f)\ge -1+\frac{1}{2}+\frac{1}{4} \cdot 2=0$.  When $v_1$ and $v_2$ are the only $5^+$- or nontriangular $4$-vertices on $f$, they must be $5^+$-vertices.  Further, $v_4$ cannot be a $4$-vertex on an internal $(3,3,4)$-face, and by (R\ref{itm:4v_5f}) and (R\ref{itm:5+v_inci5}), $\mu^*(f)\ge -1+\frac{3}{8}\cdot 2+\frac{1}{4}=0$.
\end{proof}

Clearly, each vertex on $C$ has final charge $0$, since all its (positive or negative) charges are given to $C$. Now we consider the vertices in $int(C)$.  By Lemma~\ref{No2V}, if $v\in int(C)$, then $d(v)\ge3$. If $d(v)=3$, then the initial charge on $v$ is $2d(v)-6=0$, and $v$ does not distribute charge during discharging. Hence we consider $d(v)\ge4$.

Suppose $d(v)=4$. Vertex $v$ distributes charge according to rule (R\ref{itm:4v}).  If $v$ is triangular, then it gives $2$ if it is incident with a $(3,3,4)$-face in $int(C)$, and at most $\frac{5}{4}+\frac{1}{4}\cdot 3=2$ otherwise.  If $v$ is nontriangular, then it gives at most $\frac{1}{2}\cdot 4=2$.   Hence the charge of all 4-vertices after the discharge procedure is at least 0.

\medskip

\begin{lemma}
Triangular $5^+$-vertices in $int(C)$ have nonnegative final charge.
\end{lemma}

\begin{proof}
Let $v$ be a triangular $d$-vertex in $int(C)$ with $d\ge 5$.  Let $f_0$ be the $3$-face incident with $v$.

If $f_0\cap C\not=\emptyset$,  then $v$ does not send charge to $f_0$.  Since $v$ gives at most $\frac{1}{2}$ to its other incident faces and $\frac{1}{4}$ to each of the special pendant 5-faces, $\mu^*(v)\ge (2d-6)-\frac{1}{2}\cdot (d-1)-\frac{1}{4}\cdot d>0$.  Now assume $f_0 \in int(C)$.

Let $d=5$. 
\begin{itemize}
\item $f_0$ is a $(3,4^-,5)$-face. By Lemmas~\ref{335Face} and~\ref{335Face2}, one of the following must hold: (1) $v$ is adjacent to at most one special vertex, (2) $v$ is adjacent to two special vertices and incident with a face that contains a vertex in $C$, or (3) $v$ is adjacent to two special vertices and a nontriangular $5^+$-vertex, which implies $v$ has at most two incident 5-faces with non-consecutive $5^+$-vertices. In all cases, (R2) governs the distribution of charge: in case (1), $\mu^*(v)\ge 4-\frac{7}{4}-\frac{1}{2}\cdot 4-\frac{1}{4}=0$; in case (2), $\mu^*(v)\ge 4-\frac{7}{4}-\frac{1}{2}\cdot 3-\frac{1}{4}\cdot 2>0$; and in case (3), if $v$ is a bad $5$-vertex, then $\mu^*(v)\ge 4-\frac{7}{4}-\frac{1}{2}\cdot 2-\frac{3}{8}\cdot 2-\frac{1}{4}\cdot 2=0$;   note that $v$ cannot be a good $5$-vertex, for otherwise, $f_0$ must be a $(3,3,5)$-face and the two $3$-vertices are special, but Lemma~\ref{335Face} tells that they must be adjacent to a pendant $5^+$-vertices, a contradiction.

\item  $f_0$ is a $(3, 5,5)$-face. By Lemma~\ref{335Face2}, $v$ is adjacent to at most two special vertices, and by (R2), $\mu^*(v)\ge 4-\frac{3}{2}-\frac{1}{2}\cdot 4-\frac{1}{4}\cdot 2= 0$.

\item $f_0$ is any other $3$-face. By (R\ref{itm:5v}), $v$ gives 1 to $f_0$. Lemma~\ref{Max1} implies that $v$ has a maximum of three adjacent special vertices, so by (R2), $\mu^*(v)\ge 4-1-\frac{1}{2}\cdot 4-\frac{1}{4}\cdot 3 > 0$.
\end{itemize}

Let $d=6$.  First assume that $f_0$ is a $(3^+,4^+,6)$-face. By Lemma~\ref{Max1} and (R2), $v$ gives $\frac{1}{4}$ to up to four adjacent special vertices. By (R2a), $v$ gives at most $\frac{1}{2}$ to each incident 5-face, and by (R\ref{itm:6+v}), $v$ gives at most 2 to the incident 3-face. Thus $\mu^*(v) \geq 6-(2+\frac{1}{4} \cdot 5+\frac{1}{2} \cdot 5)>0.$    So we may assume that $f_0$ is a $(3,3,6)$-face.
By Lemma~\ref{336Face}, $v$ is adjacent to no more than two special vertices or one of its neighbors is in $V(C)$  (this implies that $v$ has at most three incident $5$-faces in $int(C)$).  By (R2), in the former case, $\mu^*(v)\ge 6-3-\frac{1}{2}\cdot 5-\frac{1}{4}\cdot 2=0$, and in the latter case, $\mu^*(v)\ge 6-3-\frac{1}{2}\cdot 3-\frac{1}{4}\cdot 4>0$.

Finally, let $d\ge 7$. Then $v$ is incident with at most $d-1$ faces of length 5, and by Lemma~\ref{Max1}, at most $d-2$ special vertices.  By (R2), $\mu^*(v)\ge 2d-6-3-\frac{1}{2}\cdot (d-1)-\frac{1}{4}\cdot (d-2)=\frac{1}{4}(5d-32)>0$.
\end{proof}

In the rest of the paper, whenever mentioned, $(3,3,5^-)$-faces, $(3,4,4)$-faces, and $5$-faces are in $int(C)$.

\begin{lemma}
Nontriangular $6^+$-vertices in $int(C)$ have nonnegative final charge.
\end{lemma}

\begin{proof}
Let $v\in int(C)$ be a nontriangular $6^+$-vertex and let $t$ be the number of pendant $(3,3,5^-)$ or $(3,4,4)$-faces of $v$. By Lemma~\ref{Max1}, $v$ has at most $(d-t-2)$ pendant special $5$-faces. By (R\ref{itm:5+v}), $v$ gives at most 1 to each pendant $3$-face, at most $\frac{1}{2}$ to each incident $5$-face, and $\frac{1}{4}$ to each pendant special $5$-face. So if $t\le d-4$, then $\mu^*(v) \geq 2d-6-t-\frac{1}{2}d -\frac{1}{4}(d-t-2)=\frac{1}{4}(5d-22-3t)\ge\frac{1}{4}(2d-10)>0$. If $t=d-2$, then $v$ has no pendant special 5-faces, and $d_{\triangle}(G)\ge 2$ implies that $v$ has at most $4$ incident $5$-faces. So $\mu^*(v) \geq 2d-6-(d-2)-\frac{1}{2}(4)=d-3\ge0$. Since $t\le d-2$  by Lemma~\ref{Max1}, it remains only to check $t=d-3$.

If $d\ge7$, then $\mu^*(v) \geq (2d-6)-(d-3)-\frac{1}{2}d-\frac{1}{4}(d-t-2)=\frac{1}{2}d-\frac{13}{4}>0$. If $d=6$ and at least one of the three pendant $3$-faces of $v$ is not a $(3,3,3)$-face, then $v$ gives at most $\frac{5}{8}$ to this $3$-face by (R\ref{itm:5+v}), so $\mu^*(v) \geq 6-1\cdot 2-\frac{5}{8}-\frac{1}{2}\cdot 6-\frac{1}{4}=\frac{1}{8}>0$. So we may assume that the $6$-vertex $v$ is adjacent to exactly three pendant $(3,3,3)$-faces. If $v$ has at most five incident $5$-faces, then $\mu^*(v)\geq 6-1\cdot 3-\frac{1}{2}\cdot 5-\frac{1}{4}=\frac{1}{4}>0$. If $v$ has six incident $5$-faces, then $v$ is adjacent to no pendant special $5$-faces since by Lemma~\ref{335Face}, the outer neighbors of the $3$-vertices on a $(3,3,3)$-face either are in V(C) or have degree at least 5. So $\mu^*(v) \geq 6-1\cdot 3-\frac{1}{2}\cdot 6=0$.
\end{proof}

\begin{lemma}
Nontriangular $5$-vertices in $int(C)$ have nonnegative final charge.
\end{lemma}

\begin{proof}
Let $v\in int(C)$ be a nontriangular $5$-vertex that is adjacent to $t$ pendant $(3,3,5^-)$- or $(3,4,4)$-faces. Let $N(v)=\{v_i:  1\le i\le5\}$. Let $f_i$ be the incident face of $v$ containing $v_i,v,v_{i+1}$ for $1\le i\le5$ (index modulo $5$). By Lemma~\ref{Max1}, $v$ has at most $d-t-2$ pendant special 5-faces.  By (R\ref{itm:5+v}), $v$ gives at most 1 to each pendant $3$-face, at most $\frac{5}{8}$ to each pendant 3-face that is not a $(3,3,3)$-face, at most $\frac{1}{2}$ to each incident $5$-face, and $\frac{1}{4}$ to each pendant special $5$-face. If $t\le1$, then $v$ has at most two pendant special $5$-faces, and $\mu^*(v)\geq 4-1-\frac{1}{2}\cdot5-2\cdot\frac{1}{4}=0$. Since  $t\le3$ by Lemma~\ref{Max1}, we may assume that $t=2$ or $t=3$.

{\bf Case 1:} $t=2$.  

First we verify {\bf Claim A:}  If $v_i$ is on a $(3,3,3)$-face, both $f_i$ and $f_{i-1}$ are $5$-faces.

For otherwise, suppose $v_1$ is on a $(3,3,3)$-face and $f_1$ is not a $5$-face. Since an interior triangle of a bad $6$-face must share edges with three $5$-faces, the $(3,3,3)$-face at $v_1$ is not the interior triangle of a bad $6$-cycle.  If $f_5$ is not a 5-face, then $\mu^*(v)\ge 4-2\cdot 1+\frac{1}{2}\cdot 3+\frac{1}{4} > 0$.  If $f_5$ is a $5$-face, then it shares an edge with a $(3,3,3)$-face, thus by Lemma~\ref{5Face}, since $v$ has two pendant $(3,3,5^-)$- or $(3,4,4)$-faces, it has no pendant special $5$-face.  By (R2), $\mu^*(v)\ge 4-2\cdot 1+\frac{1}{2}\cdot 4 = 0$. Hence Claim A is established.

Without loss of generality, we may assume that either $v_1$ and $v_2$ or $v_1$ and $v_3$ are on $(3,3,5^-)$- or $(3,4,4)$-faces.  
First let $v_1$ and $v_2$ be on $(3,3,5^-)$- or $(3,4,4)$-faces. Note that $f_1$ is a $6^+$-face since $d_{\triangle}(G)\ge 2$.  By Claim A, neither $v_1$ nor $v_2$ is on a $(3,3,3)$-face. Hence by (R2), $\mu^*(v)\geq 4-\frac{5}{8}\cdot 2-\frac{1}{2}\cdot 4-\frac{1}{4}=\frac{1}{2}>0$.  


Now let $v_1, v_3$ be on $(3,3,5^-)$- or $(3,4,4)$-faces. Note that $v$ is a bad $5$-vertex. We first suppose that one of $f_1,f_2$ (say $f_1$) is not a $5$-face.  By Claim A, $v_1$ is not on a $(3,3,3)$-face. So by (R2),  $\mu^*(v)\geq 4-1-\frac{5}{8}-\frac{1}{2}\cdot 4-\frac{1}{4}=\frac{1}{8}>0$.


Now assume that both $f_1$ and $f_2$ are $5$-faces. By Lemma~\ref{two3face}, $d(v_2)\ge 4$. If $v$ has no pendant $(3,3,3)$-faces, then $\mu^*(v)\geq 4-\frac{5}{8}\cdot 2-\frac{1}{2}\cdot 5-\frac{1}{4}=0$. Hence we may assume that $v_1$ is on a $(3,3,3)$-face. By Claim A, $f_5$ is a $5$-face. By Lemma~\ref{335Face}, the pendant neighbors (in particular, on $f_1$ and $f_5$) of the $(3,3,3)$-faces are in $V(C)$ or have degree at least 5.

First consider the case that $v_3$ is also on a $(3,3,3)$-face.  By Claim A, $f_3$ is also a $5$-face, and the pendant neighbors (in particular, on $f_2$ and $f_3$) of the $(3,3,3)$-faces are in $V(C)$ or have degree at least 5. Note that $v_2$ cannot be triangular; further, three of its consecutive neighbors are $5^+$-vertices.  Hence $v_2$ is a good $4^+$-vertex and both $f_1$ and $f_3$ are rich $5$-faces. By (R\ref{itm:5+v_inci5}), $v$ gives $\frac{1}{4}$ to each of $f_1$ and $f_2$.   Note that both $v_4$ and $v_5$ are next to a $5^+$-neighbor respectively on $f_3$ and $f_5$, so they are not pendant special $3$-vertices of $v$, and thus $v$ has no pendent special $5$-faces. So by (R2),  $\mu^*(v)\geq 4-1\cdot2-\frac{1}{4}\cdot2-\frac{1}{2}\cdot3=0$.


Finally, assume that $v_3$ is not on a $(3,3,3)$-face. 
Suppose $v_3$ is on a $(3,4^-,4)$-face. If $d(v_2)\not=5$, then $v_2$ is a good $4^+$-vertex and $f_1$ is rich, so by (R2),  $\mu^*(v)\geq 4-1-\frac{1}{2}-\frac{1}{4}-\frac{1}{2}\cdot4-\frac{1}{4}=0$; if $d(v_2)=5$, then $\mu^*(v)\geq 4-1-\frac{1}{2}-\frac{1}{3}-\frac{3}{8}-\frac{1}{2}\cdot3-\frac{1}{4}=\frac{1}{24}>0.$   So we assume that $v_3$ is on a $(3,3,5)$-face. Then by Lemma~\ref{335Face}, $f_2$ contains at least three $5^+$- or good $4^+$-vertices. So $v_2$ is a good $4^+$-vertex and both $f_1$ and $f_2$ are rich $5$-faces. By (R\ref{itm:5+v_inci5}), $v$ gives $\frac{1}{4}$ to $f_1$ and $f_2$.  Then $\mu^*(v)\geq 4-1-\frac{5}{8}-\frac{1}{4}\cdot2-\frac{1}{2}\cdot3-\frac{1}{4}=\frac{1}{8}>0$.

{\bf Case 2:} $t=3$.

By Lemma~\ref{Max1}, $v$ has no pendant special $5$-faces. Without loss of generality, we may assume that either $v_1,v_2,v_3$ or $v_1,v_2,v_4$ are on pendant $(3,3,5^-)$- or $(3,4,4)$-faces.

Assume first that $v_1,v_2,v_3$ are on pendant $(3,3,5^-)$- or $(3,4,4)$-faces.  Since $d_{\triangle}(G)\ge 2$, neither $f_1$ nor $f_2$ is a $5$-face; hence $v$ has at most three incident $5$-faces and is not on a bad $6$-cycle.   If $v$ has at most two incident $5$-faces, then $\mu^*(v)\geq 4-1\cdot3-\frac{1}{2}\cdot2=0$. Thus we suppose that $f_3$ and $f_5$ are $5$-faces (with an incident $3$-face). By Lemma~\ref{5Face}, one of the neighbors of $v_1$ on its triangular face must have degree at least 4; the same is true for $v_3$.  Hence neither of the pendant triangles at $v_1$ and $v_3$ are $(3,3,3)$-faces. By (R2), $\mu^*(v)\geq 4-1-\frac{5}{8}\cdot2-\frac{1}{2}\cdot3=\frac{1}{4}>0$.

Now assume that $v_1,v_2,v_4$ are on pendant $(3,3,5^-)$ or $(3,4,4)$-faces.  Since $d_{\triangle}(G)\ge 2$, $f_1$ is not a $5$-face. Then the $3$-face at $v_2$ is not an interior triangle of a bad $6$-cycle.  By Lemma~\ref{5Face}, either $f_2$ is a $5$-face and $v_2$ is not on a $(3,3,3)$-face or $f_2$ is not a $5$-face.  Hence $v$ gives at most $\max\{1,\frac{5}{8}+\frac{1}{2}\}=\frac{9}{8}$ to $f_2$ and the pendant $3$-face at $v_2$ by (R\ref{itm:5+v}). Similarly, $v$ gives at most $\frac{9}{8}$ to $f_5$ and the pendant $3$-face at $v_1$. If $v_4$ is not on a $(3,3,3)$-face, then $\mu^*(v)\geq 4-\frac{9}{8}\cdot2-\frac{1}{2}\cdot2-\frac{5}{8}=\frac{1}{8}>0.$

Hence we may assume that $v_4$ is on a $(3,3,3)$-face.   If neither $f_2$ nor $f_5$ is a  $5$-face, then $\mu^*(v)\geq 4-1\cdot3-\frac{1}{2}\cdot2=0$, so we may also assume (by symmetry) that  $f_5$ is a $5$-face. If one of $f_3,f_4$ is not a $5$-face, then $\mu^*(v)\geq 4-1-\frac{9}{8}\cdot2-\frac{1}{2}=\frac{1}{4}>0$, so we may assume that both $f_3$ and $f_4$ are $5$-faces.    Since $f_1$ is not a $5$-face, the $3$-face at $v_1$ is not the interior $3$-face of a bad $6$-cycle, so by Lemma~\ref{5Face}, it cannot be a $(3,3,3)$-face.  Now that both $f_4,f_5$ are $5$-faces, by Lemma~\ref{two3face}, $d(v_5)\ge4$.     By Lemma~\ref{335Face}, the pendant neighbors of the $(3,3,3)$-face at $v_4$ are in $V(C)$ or have degree at least 5, hence $f_4$ has at least three $4^+$-vertices.   Since $d_{\triangle}(G)\ge 2$, $v_4$ cannot be a triangular $4$-vertex. If $d(v_5)\ne5$, then $v_4$ is a good $4^+$-vertex and $f_4$ is a rich $5$-face; by (R2), $\mu^*(v)\geq 4-\frac{9}{8}\cdot2-1-\frac{1}{4}-\frac{1}{2}=0$. So let $d(v_5)=5$. If $v_1$ is on a $(3,4^-,4)$-face, then $\mu^*(v)\geq 4-1-\frac{1}{2}-\frac{9}{8}-\frac{1}{2}\cdot2-\frac{1}{3}=\frac{1}{24}>0$. If $v_1$ is on a $(3,3,5)$-face, then the $5$-vertex must be on $f_5$ by Lemma~\ref{5Face}. So by (R\ref{itm:5+v}), $v$ gives $\frac{1}{3}$ to each of $f_4$ and $f_5$. Then $\mu^*(v)\geq 4-1-\frac{5}{8}-\frac{1}{3}\cdot2-\frac{9}{8}-\frac{1}{2}=\frac{1}{12}>0$.
\end{proof}

Therefore all vertices have nonnegative charge after the discharge procedure.

\end{document}